\documentclass[10pt,a4paper]{article}
\usepackage{amsfonts}
\usepackage[latin9]{inputenc}
\usepackage{amsmath}
\usepackage{amssymb}
\usepackage{url}
\usepackage{graphicx}
\usepackage[pdfstartview=FitH]{hyperref}
\usepackage{amsthm}
\usepackage{authblk}
\makeatletter
\begin{document}
\newtheorem{theorem}{Theorem}[section]
\newtheorem{lemma}[theorem]{Lemma}
\newtheorem{definition}[theorem]{Definition}
\newtheorem{claim}[theorem]{Claim}
\newtheorem{example}[theorem]{Example}
\newtheorem{remark}[theorem]{Remark}
\newtheorem{proposition}[theorem]{Proposition}
\newtheorem{corollary}[theorem]{Corollary}

\title{An intermediate quasi-isometric invariant between subexponential asymptotic dimension growth and Yu's Property A}
\author{Izhar Oppenheim}
\affil{Department of Mathematics\\
 The Ohio State University  \\
 Columbus, OH 43210, USA \\
E-mail: izharo@gmail.com}

\maketitle
\textbf{Abstract}. We present the notion of asymptotically large depth for a metric space which is (a priory) weaker than having subexponential asymptotic dimension growth and (a priory) stronger than property A.    \\ \\
\textbf{Mathematics Subject Classification (2010)}. 51F99, 20F65 
\textbf{Keywords}.  Property A, Asymptotic dimension growth

\section{Introduction}

In \cite{Grom} Gromov introduced the notion of finite asymptotic dimension to study infinite groups. This notion gained more popularity when in \cite{Yu2} Yu showed that groups with finite asymptotic dimension satisfy the Novikov conjuncture. Later, in \cite{Yu}, Yu defined a weaker notion called property A (which can be viewed as a non equivariant notion of amenability) and proved that groups with property A also satisfy the Novikov conjuncture. It has been shown that finite asymptotic dimension imply property A (see for instance \cite{Will}[Theorem 1.2.4]) and that there are groups with infinite asymptotic dimension for which property A holds (see for instance \cite{Dr}[section 4] to see that for the restricted wreath product, $\mathbb{Z} \wr \mathbb{Z}$, the asymptotic dimension is infinite and property A holds). \\
 
Also, in \cite{Dr} the notion of polynomial asymptotic growth was discussed and it was shown that polynomial asymptotic growth implies property A. Later, in \cite{Oz}, it was shown that subexponential asymptotic growth implies property A. In this article, we shall introduce another notion that we name asymptotically large depth. We'll show having asymptotically large depth implies property A and is implied by than having subexponential asymptotic growth (thus providing another proof to the fact that subexponential asymptotic growth implies property A). We have no new examples of using this notion so far, but it is our hope that using this notion of asymptotically large depth will open a new way of proving property A for finitely generated groups. \\

\textbf{Structure of this paper}: In section 2, we recall basic definitions regarding property A, give a first definition of asymptotically large depth and show that it implies property A for discrete metric spaces of bounded geometry. In section 3, we prove equivalent definitions for asymptotically large depth that are maybe more intuitive. In section 4, we show that subexponential asymptotic growth implies asymptotically large depth. In section 5, we recall some definitions regarding quasi-isometry and show that asymptotically large depth is a quasi-isometric invariant.

\section{Asymptotically large depth - first definition and property A}

\subsection{Property A}

Let us start by recalling some basic definition regarding discrete metric spaces:

\begin{definition}
Let $(X,d)$ be a metric space. \\
$X$ is called discrete if there is $c>0$ such that 
$$ \forall x,y \in X, d(x,y) < c \Rightarrow x=y.$$
A discrete metric space $X$ is said to have bounded geometry if 
$$\forall M>0, \sup_{x \in X} \vert B(x,M) \vert < \infty .$$ 
\end{definition}

\begin{remark}
The motivating example of discrete metric spaces with bounded geometry is $X = \Gamma$ a finitely generated group with the word metric with respect to some finite generating set $S$. 
\end{remark}

Property A was defined in \cite{Yu} for general discrete metric spaces. We shall focus our attention on discrete metric spaces of bounded geometry and therefore we shall not state the original definition here, but only an equivalent definition for discrete metric spaces of bounded geometry (the equivalence to the original definition is proven in \cite{HR}[Lemma 3.5] or \cite{Will}[Theorem 1.2.4]).

\begin{definition}
Let $(X,d)$ be a discrete metric space with bounded geometry. $X$ has property A if there is a family of functions $\lbrace a_k^x \in l^1 (X) \rbrace_{x \in X}$ such that the following holds (we denote the norm of  $ l^1 (X)$ by $\Vert . \Vert$):
\begin{enumerate}
\item For every $k$ there is $S_k >0$ such that for every $x$, $a_k^x$ is supported on $B(x,S_k)$. 
\item For every $R>0$ we have that
$$\lim_{k \rightarrow \infty} \sup_{x,y \in X, d(x,y) <R} \dfrac{\Vert a_k^x - a_k^y \Vert}{\Vert a_k^x \Vert}  =0.$$
\end{enumerate} 
\end{definition}

\subsection{Asymptotically large depth and connection to property}
Given a metric space $(X,d)$ and $\mathcal{U}$ a cover of $X$. Recall that $\mathcal{U}$ is called uniformly bounded if there is $S_\mathcal{U} >0$, such that $\forall U \in \mathcal{U}, diam (U) <S_\mathcal{U}$. \\
For $x \in X$ will define $m_\mathcal{U} (x)$ to be the multiplicity of $\mathcal{U}$ at $x$, i.e.,:
$$m_\mathcal{U} (x) =  \vert \lbrace U \in \mathcal{U}: x \cap U \neq \emptyset \rbrace \vert$$
$\mathcal{U}$ will be called locally finite if $m_\mathcal{U} (x) < \infty$ for every $x \in X$. \\
\begin{definition}
Let $(X,d)$ be a metric space. We shall say that $X$ has asymptotically large depth if it has a sequence of covers $\lbrace \mathcal{U}_k\rbrace_{k \in \mathbb{N}}$, such that:
\begin{enumerate}
\item  For each $k$, $\mathcal{U}_k$ is uniformly bounded and locally finite.
\item There is $\varepsilon >0$ and a function $f: \mathbb{N} \rightarrow \mathbb{R}_+$ with $\lim f(k) = \infty$ such that for every $k$,
$$ \inf_{x \in X} \left( \frac{1}{m_{\mathcal{U}_k} (x)}\sum_{U \in \mathcal{U}_k, x \in U} e^{\frac{d(x,X \setminus U)}{f(k)}} \right) \geq 1 + \varepsilon.$$
\end{enumerate}  
\end{definition}

\begin{remark}
In the case in which $X$ is bounded and $U =X$, define $d(x,X \setminus U) = d(x, \emptyset) = \infty$. This definition takes care of the trivial case in which $X$ is bounded. 
\end{remark}

\begin{remark}
The motivation behind the above definition comes from the frequent use of the distance $d(x,X \setminus U)$ in proofs of properties that imply property A. For example, one can see $d(x,X \setminus U)$ appearing in the proof that finite asymptotic dimension implies property A in \cite{Will}[proposition 2.2.6] or in the proof that subexponential asymptotic dimension growth implies property A in \cite{Oz}). The main theme of such proofs is that if one can find a sequence of covers such that each point $x$ is "deep enough" in enough sets (i.e., $d(x,X \setminus U)$ is large enough in enough sets) then we get property A. The idea behind the definition above is to try to take this idea to the extreme and find a very weak notion of each point $x$ being "deep enough". The above definition might not seem very enlightening - a cleaner and maybe more intuitive version is given in theorem \ref{subexp function} below. 
\end{remark}

\begin{theorem}
\label{property A}
Let $(X,d)$ be a discrete metric space with bounded geometry. If $X$ has asymptotically large depth then it has property A.
\end{theorem}

\begin{proof}
We shall construct a family of functions $\lbrace a_k^x \in l^1 (X) \rbrace_{x \in X}$ as in the definition of property A stated above. \\
For $k \in \mathbb{N}$, choose $x_U \in U$ for every $U \in \mathcal{U}_k$ and denote $\delta_U = \delta_{x_U} \in l^1 (X)$. Define $a_k^x : X \rightarrow l^1 (X)$ as following:
$$a_k^x  = \sum_{U \in \mathcal{U}_k} \left( e^{\frac{d(x,X \setminus U)}{f(k)}} -1 \right) \delta_U.$$ 
First, we'll show that we can take $S_k$ to be the uniform bound on the diameter of the sets in $\mathcal{U}_k$. Indeed, if $x \notin U$, then $d(x,X \setminus U) =0$ and therefore $e^{\frac{d(x,X \setminus U)}{f(k)}} -1 =e^0-1=0$. This yields that $a_k^x$ is supported on the set 
$$\bigcup_{U \in \mathcal{U}_k, x \in U} U \subseteq B(x, S_k).$$ 
Second, for any $R>0$ and any $x,y \in X$ with $d(x,y) <R$ we have that 
$$a_k^x - a_k^y = \sum_{U \in \mathcal{U}_k} \left( e^{\frac{d(x,X \setminus U)}{f(k)}} - e^{\frac{d(x,X \setminus U)}{f(k)}} \right) \delta_U.$$
If $d(x,X \setminus U) > d(y,X \setminus U) \geq 0$, then
$$0 \leq e^{\frac{d(x,X \setminus U)}{f(k)}} - e^{\frac{d(x,X \setminus U)}{f(k)}} = e^{\frac{d(x,X \setminus U)}{f(k)}} \left( 1- e^{\frac{d(x,X \setminus U) - d(x,X \setminus U)}{f(k)}} \right) \leq e^{\frac{d(x,X \setminus U)}{f(k)}} \left( 1 - e^\frac{-R}{f(k)} \right).$$
Therefore, by symmetry, we get that 
$$\Vert a_k^x - a_k^y \Vert \leq \left( 1 - e^\frac{-R}{f(k)} \right) \left( \sum_{U \in \mathcal{U}_k, x \in U} e^{\frac{d(x,X \setminus U)}{f(k)}} + \sum_{U \in \mathcal{U}_k, y \in U} e^{\frac{d(y,X \setminus U)}{f(k)}} \right).$$
Without loss of generality, we can assume
$$\sum_{U \in \mathcal{U}_k, x \in U} e^{\frac{d(x,X \setminus U)}{f(k)}} \geq \sum_{U \in \mathcal{U}_k, y \in U} e^{\frac{d(y,X \setminus U)}{f(k)}},$$
and therefore
$$\Vert a_k^x - a_k^y \Vert \leq 2 \left( 1 - e^\frac{-R}{f(k)} \right) \left( \sum_{U \in \mathcal{U}_k, x \in U} e^{\frac{d(x,X \setminus U)}{f(k)}}\right).$$
Notice that
$$\Vert a_k^x \Vert = \sum_{U \in \mathcal{U}_k, x \in U} \left( e^{\frac{d(x,X \setminus U)}{f(k)}} - 1 \right) = \left( \sum_{U \in \mathcal{U}_k, x \in U} e^{\frac{d(x,X \setminus U)}{f(k)}} \right) - m_{\mathcal{U}_k} (x) . $$
Therefore we get
$$\dfrac{\Vert a_k^x - a_k^y \Vert}{\Vert a_k^x \Vert } \leq \dfrac{2 \left( 1 - e^\frac{-R}{f(k)} \right) \left( \sum_{U \in \mathcal{U}_k, x \in U} e^{\frac{d(x,X \setminus U)}{f(k)}}\right)}{\left( \sum_{U \in \mathcal{U}_k, x \in U} e^{\frac{d(x,X \setminus U)}{f(k)}} \right) - m_{\mathcal{U}_k} (x)} = $$
$$ = \dfrac{2 \left( 1 - e^\frac{-R}{f(k)} \right)}{1 - \dfrac{m_{\mathcal{U}_k} (x)}{\left( \sum_{U \in \mathcal{U}_k, x \in U} e^{\frac{d(x,X \setminus U)}{f(k)}} \right) }} \leq \dfrac{2 \left( 1 - e^\frac{-R}{f(k)} \right)}{1 -\frac{1}{1+\varepsilon}}= \dfrac{2 + 2 \varepsilon}{\varepsilon} \left( 1 - e^\frac{-R}{f(k)} \right),$$
where the last inequality is due to the condition in the definition of asymptotically large depth. The computation above yields that for every $k$ and every $R>0$ we have that
$$\sup_{x,y \in X, d(x,y) <R} \dfrac{\Vert a_k^x - a_k^y \Vert}{\Vert a_k^x \Vert } \leq \dfrac{2 + 2 \varepsilon}{\varepsilon} \left( 1 - e^\frac{-R}{f(k)} \right)  .$$
Recall that $\lim f(k) = \infty$ and therefore for every $R>0$, we have that
$$\lim_{k \rightarrow \infty} \sup_{x,y \in X, d(x,y) <R}  \dfrac{\Vert a_k^x - a_k^y \Vert}{\Vert a_k^x \Vert } =0,$$
which finishes the proof.
\end{proof}

\section{Equivalent definitions}
\begin{theorem}
\label{definitions}
Let $(X,d)$ a metric space. The following are equivalent:
\begin{enumerate}
\item $X$ has asymptotically large depth.
\item $X$ has a sequence of covers $\lbrace \mathcal{U}_k\rbrace_{k \in \mathbb{N}}$, such that:
\begin{enumerate}
\item  For each $k$, $\mathcal{U}_k$ is uniformly bounded and locally finite.
\item There is $\varepsilon >0$ and a sequence $\lbrace c_k \rbrace_{k \in \mathbb{N}}$ with $c_k > 1, \forall k$ and $\lim c_k =1$ such that for every $k$
$$ \inf_{x \in X} \left( \frac{1}{m_{\mathcal{U}_k} (x)}\sum_{U \in \mathcal{U}_k, x \in U} c_k^{d(x,X \setminus U)} \right) \geq 1 + \varepsilon.$$
\end{enumerate}  
\item There is $\varepsilon >0$ such that for any number $c>1$, $X$ has a cover $\mathcal{U} (c)$, such that:
\begin{enumerate}
\item $\mathcal{U} (c)$ is uniformly bounded and locally finite.
\item
$$ \inf_{x \in X} \left( \frac{1}{m_{\mathcal{U} (c)} (x)}\sum_{U \in \mathcal{U} (c), x \in U} c^{d(x,X \setminus U)} \right) \geq 1 + \varepsilon.$$
\end{enumerate} 
\item $X$ has a sequence of covers $\lbrace \mathcal{U}_k\rbrace_{k \in \mathbb{N}}$, such that:
\begin{enumerate}
\item  For each $k$, $\mathcal{U}_k$ is uniformly bounded and locally finite.
\item There is a function $f' : \mathbb{N} \rightarrow \mathbb{R}_+$ with $\lim f' (k) = \infty$ such that for every $k$,
$$\lim_{k \rightarrow \infty} \inf_{x \in X} \left( \frac{1}{m_{\mathcal{U}_k} (x)}\sum_{U \in \mathcal{U}_k, x \in U} e^{\frac{d(x,X \setminus U)}{f' (k)}} \right) = \infty.$$
\end{enumerate}  
\item $X$ has a sequence of covers $\lbrace \mathcal{U}_k\rbrace_{k \in \mathbb{N}}$, such that:
\begin{enumerate}
\item  For each $k$, $\mathcal{U}_k$ is uniformly bounded and locally finite.
\item There is a sequence $\lbrace c_k \rbrace_{k \in \mathbb{N}}$ with $c_k > 1, \forall k$ and $\lim c_k =1$ such that 
$$\lim_{k \rightarrow \infty} \inf_{x \in X} \left( \frac{1}{m_{\mathcal{U}_k} (x)}\sum_{U \in \mathcal{U}_k, x \in U} c_k^{d(x,X \setminus U)} \right) =\infty.$$
\end{enumerate}  
\item For any $k>0$ and any $c>1$, $X$ has a cover $\mathcal{U} (c,k)$, such that:
\begin{enumerate}
\item $\mathcal{U} (c,k)$ is uniformly bounded and locally finite.
\item
$$ \inf_{x \in X} \left( \frac{1}{m_{\mathcal{U} (c,k)} (x)}\sum_{U \in \mathcal{U} (c,k), x \in U} c^{d(x,X \setminus U)} \right) \geq k.$$
\end{enumerate} 
\end{enumerate}
\end{theorem}

\begin{proof}
$(1) \Rightarrow (2)$: Let $f: \mathbb{N} \rightarrow \mathbb{R}_+$ as in the definition of asymptotically large depth, and take $c_k = e^{\frac{1}{f(k)}}$. Since $f(k) >0$ and $\lim f(k) = \infty$ we get that $c_k >1$ and $\lim c_k =1$. We finish by taking $\lbrace \mathcal{U}_k \rbrace_{k \in \mathbb{N}}$ to be as in the definition of asymptotically large depth and noticing that for every $k$,
 $$\inf_{x \in X} \left( \frac{1}{m_{\mathcal{U}_k}} \sum_{U \in \mathcal{U}_k, x \in U} c_k^{d(x,X \setminus U)}  \right) = \inf_{x \in X} \left( \frac{1}{m_{\mathcal{U}_k}} \sum_{U \in \mathcal{U}_k, x \in U} e^{\frac{d(x,X \setminus U)}{f (k)}} \right) \geq 1 + \varepsilon.$$
$(2) \Rightarrow (3)$: Take $\varepsilon>0$ to be the same as in $(2)$. Let $c>1$ and let $\lbrace c_k \rbrace_{k \in \mathbb{N}}$ be the sequence in $(2)$. Since $\lim c_k =1$, there is some $k_0$ such that $c>c_{k_0}$. Take $\mathcal{U} (c) = \mathcal{U}_{k_0}$ and finish by 
$$\inf_{x \in X} \left( \frac{1}{m_{\mathcal{U} (c)} (x)}\sum_{U \in \mathcal{U} (c), x \in U} c^{d(x,X \setminus U)} \right)  \geq \inf_{x \in X} \left( \frac{1}{m_{\mathcal{U}_{k_0}} (x)}\sum_{U \in \mathcal{U}_{k_0}, x \in U} c_{k_0}^{d(x,X \setminus U)} \right) \geq 1 + \varepsilon.$$
$(3) \Rightarrow (1)$: Take $\varepsilon>0$ to be the same as in $(3)$. For every $k \in \mathbb{N}$, take $\mathcal{U}_k = \mathcal{U} (e^\frac{1}{k})$ and $f(k)=k$. Finish by 
$$\inf_{x \in X} \left( \frac{1}{m_{\mathcal{U}_k (x)}} \sum_{U \in \mathcal{U}_k, x \in U} e^{\frac{d(x,X \setminus U)}{f(k)}} \right)  = \inf_{x \in X} \left( \frac{1}{m_{\mathcal{U} (e^\frac{1}{k})} (x)}\sum_{U \in \mathcal{U} (e^\frac{1}{k}), x \in U} (e^\frac{1}{k})^{d(x,X \setminus U)} \right)  \geq 1 + \varepsilon.$$
$(4) \Rightarrow (5), (5) \Rightarrow (6), (6) \Rightarrow (4)$: Same proofs as $(1) \Rightarrow (2), (2) \Rightarrow (3), (3) \Rightarrow (1)$.\\
$(4) \Rightarrow (1)$: Obvious. \\
$(1) \Rightarrow (4)$: Take $\lbrace \mathcal{U}_k\rbrace_{k \in \mathbb{N}}$ as in $(1)$ and $f' (k) = \sqrt{f(k)}$. For $x \in X$ denote:
$$A (x) = \lbrace U \in \mathcal{U}_k : x \in U, e^{\frac{d(x,X \setminus U)}{f (k)}} \geq 1 + \frac{\varepsilon}{2} \rbrace,$$
$$B (x) = \lbrace U \in \mathcal{U}_k : x \in U, e^{\frac{d(x,X \setminus U)}{f (k)}} < 1 + \frac{\varepsilon}{2} \rbrace .$$
Since for every $x$ we have by assumption
$$ \left( \frac{1}{m_{\mathcal{U}_k} (x)}\sum_{U \in \mathcal{U}_k, x \in U} e^{\frac{d(x,X \setminus U)}{f(k)}} \right) \geq 1 + \varepsilon,$$
we get that $A(x)$ is not empty. For any $x \in X$ we have that:
$$ \frac{1}{m_{\mathcal{U}_k} (x)}\sum_{U \in \mathcal{U}_k, x \in U} e^{\frac{d(x,X \setminus U)}{f' (k)}} = \frac{1}{m_{\mathcal{U}_k} (x)}\sum_{U \in \mathcal{U}_k, x \in U} (e^{\frac{d(x,X \setminus U)}{f (k)}})^{\sqrt{f(k)}} \geq$$
$$\geq \frac{1}{m_{\mathcal{U}_k} (x)}\sum_{U \in A(x)} (e^{\frac{d(x,X \setminus U)}{f (k)}})^{\sqrt{f(k)}} = \frac{1}{m_{\mathcal{U}_k} (x)}\sum_{U \in A(x)} (e^{\frac{d(x,X \setminus U)}{f (k)}}) (e^{\frac{d(x,X \setminus U)}{f (k)}})^{\sqrt{f(k)}-1}  \geq$$
$$\geq \frac{1}{m_{\mathcal{U}_k} (x)}\sum_{U \in A(x)} (e^{\frac{d(x,X \setminus U)}{f (k)}}) (1+ \frac{\varepsilon}{2})^{\sqrt{f(k)}-1} \geq $$
$$\geq (1+ \frac{\varepsilon}{2})^{\sqrt{f(k)}-1} \left( \frac{1}{m_{\mathcal{U}_k} (x)}\sum_{U \in \mathcal{U}_k, x \in U} (e^{\frac{d(x,X \setminus U)}{f (k)}}) - \frac{1}{m_{\mathcal{U}_k} (x)}\sum_{U \in B(x)} (e^{\frac{d(x,X \setminus U)}{f (k)}}) \right) \geq $$
$$\geq (1+ \frac{\varepsilon}{2})^{\sqrt{f(k)}-1} ( 1+ \varepsilon - (1+\frac{\varepsilon}{2})) = \frac{\varepsilon}{2} (1+ \frac{\varepsilon}{2})^{\sqrt{f(k)}-1}$$
Since this is true for any $x \in X$, we get that
$$\inf_{x \in X} \left( \frac{1}{m_{\mathcal{U}_k} (x)}\sum_{U \in \mathcal{U}_k, x \in U} e^{\frac{d(x,X \setminus U)}{f' (k)}} \right) \geq \frac{\varepsilon}{2} (1+ \frac{\varepsilon}{2})^{\sqrt{f(k)}-1},$$
and since $\lim f(k) = \infty$, we get
$$\lim_{k \rightarrow \infty} \inf_{x \in X} \left( \frac{1}{m_{\mathcal{U}_k} (x)}\sum_{U \in \mathcal{U}_k, x \in U} e^{\frac{d(x,X \setminus U)}{f' (k)}} \right) = \infty.$$

\end{proof}

Recall that a monotone increasing function $g:\mathbb{R}_{\geq 0} \rightarrow  \mathbb{R}_+$ is said to have subexponential growth if for every $b>1$,
$$\lim_{t \rightarrow  \infty} \dfrac{g(t )}{b^t} = 0,$$
or equivalently, if for every $b>1$
$$\lim_{t \rightarrow  \infty} \dfrac{b^t}{g(t )} = \infty .$$

\begin{theorem}
\label{subexp function}
Let $(X,d)$ be a metric space. $X$ has asymptotically large depth if and only if there is a monotone increasing function $g: \mathbb{R}_{\geq 0} \rightarrow  \mathbb{R}_{+}$ with subexponential growth and a sequence of covers $\lbrace \mathcal{U}_k\rbrace_{k \in \mathbb{N}}$ of $X$ such that:
\begin{enumerate}
\item  For each $k$, $\mathcal{U}_k$ is uniformly bounded and locally finite.
\item 
$$\lim_{k \rightarrow \infty} \inf_{x \in X} \left( \frac{1}{m_{\mathcal{U}_k} (x)}\sum_{U \in \mathcal{U}_k, x \in U} g(d(x,X \setminus U)) \right) =\infty.$$
\end{enumerate}
\end{theorem}

\begin{proof}
Assume that $X$ has asymptotically large depth, then by $(5)$ in theorem \ref{definitions} we have that $X$ has a sequence of covers $\lbrace \mathcal{U}_k\rbrace_{k \in \mathbb{N}}$, such that:
\begin{enumerate}
\item  For each $k$, $\mathcal{U}_k$ is uniformly bounded and locally finite.
\item There is a sequence $\lbrace c_k \rbrace_{k \in \mathbb{N}}$ with $c_k > 1, \forall k$ and $\lim c_k =1$ such that 
$$\lim_{k \rightarrow \infty} \inf_{x \in X} \left( \frac{1}{m_{\mathcal{U}_k} (x)}\sum_{U \in \mathcal{U}_k, x \in U} c_k^{d(x,X \setminus U)} \right) =\infty.$$
\end{enumerate}  
Denote by $S_k$ the bound of the diameters of the sets in $\mathcal{U}_k$. Notice that for every $U \in \mathcal{U}_k$ and for every $x \in X$ we have $S_k \geq d(x, X \setminus U)$. By 
$$\lim_{k \rightarrow \infty} \inf_{x \in X} \left( \frac{1}{m_{\mathcal{U}_k} (x)}\sum_{U \in \mathcal{U}_k, x \in U} c_k^{d(x,X \setminus U)} \right) =\infty, $$
we get that $\lim S_k = \infty$. By passing to a subsequence we can assume that $\lbrace S_k \rbrace_{k \in \mathbb{N}}$ is strictly monotone increasing and that $\lbrace c_k \rbrace_{k \in \mathbb{N}}$ is monotone decreasing. 
Define
$$g(t) = \begin{cases}
c_1^t & t \in [0,S_1] \\
g(S_{k-1}) + c_k^{S_{k-1} + t} -  c_k^{S_{k-1}} & S_{k-1} < t \leq S_k 
\end{cases} . $$ 
$g(t)$ is monotone increasing and since $\lbrace c_k \rbrace$ is monotone decreasing we have that:
$$\forall t \geq S_{k-1}, g(t) \leq g(S_{k-1}) + c_k^{S_{k-1} + t} - c_k^{S_{k-1}},$$
and
$$\forall t \leq S_{k}, g(t) \geq c_k^t .$$
For every $b >1$, there is $k$ such that $b > c_k$ and therefore for every $t \geq S_{k-1}$ we have
$$\dfrac{g(t)}{b^t} \leq \dfrac{g(S_{k-1}) + c_k^{S_{k-1} + t} - c_k^{S_{k-1}}}{b^t} = \dfrac{g(S_{k-1}) - c_k^{S_{k-1}}}{b^t} + c_k^{S_{k-1}} \left( \dfrac{c_k}{b} \right)^t,$$
and therefore, 
$$\lim_{t \rightarrow \infty} \dfrac{g(t)}{b^t} =0 .$$
Notice that for every $k$ we have
$$\inf_{x \in X} \left( \frac{1}{m_{\mathcal{U}_k} (x)}\sum_{U \in \mathcal{U}_k, x \in U} g(d(x,X \setminus U)) \right) \geq \inf_{x \in X} \left( \frac{1}{m_{\mathcal{U}_k} (x)}\sum_{U \in \mathcal{U}_k, x \in U} c_k^{d(x,X \setminus U)} \right).$$
(again, we use the fact that for every $U \in \mathcal{U}_k$ and for every $x \in X$ we have $S_k \geq d(x, X \setminus U)$). \\
So we get 
$$\lim_{k \rightarrow \infty} \inf_{x \in X} \left( \frac{1}{m_{\mathcal{U}_k} (x)}\sum_{U \in \mathcal{U}_k, x \in U} g(d(x,X \setminus U)) \right) =\infty.$$
In the other direction, assume that there is a monotone increasing function $g: \mathbb{R}_{\geq 0} \rightarrow  \mathbb{R}_{+}$ with subexponential growth and a sequence of covers $\lbrace \mathcal{U}_k\rbrace_{k \in \mathbb{N}}$ of $X$ such that:
\begin{enumerate}
\item  For each $k$, $\mathcal{U}_k$ is uniformly bounded and locally finite.
\item 
$$\lim_{k \rightarrow \infty} \inf_{x \in X} \left( \frac{1}{m_{\mathcal{U}_k} (x)}\sum_{U \in \mathcal{U}_k, x \in U} g(d(x,X \setminus U)) \right) =\infty.$$
\end{enumerate}
Note that by the above condition, we have that $\lim_{t \rightarrow \infty}  g(t) = \infty$. Let $\lbrace T_k \rbrace$ be a monotone increasing sequence such that $\lim T_k = \infty$ and such that for every $k$, we have
$$\forall t \geq T_k, g (t) \leq (1+\frac{1}{k})^t .$$
By passing to a subsequence of $\lbrace \mathcal{U}_k\rbrace_{k \in \mathbb{N}}$, we can assume that for every $k$, we have
$$ \inf_{x \in X} \left( \frac{1}{m_{\mathcal{U}_k} (x)}\sum_{U \in \mathcal{U}_k, x \in U} g(d(x,X \setminus U)) \right) \geq 2 g(T_k) .$$
Note that
$$\inf_{x \in X} \left( \frac{1}{m_{\mathcal{U}_k} (x)}\sum_{U \in \mathcal{U}_k, d(x,X \setminus U) \geq T_k} g(d(x,X \setminus U)) \right) \geq$$
$$\geq  \inf_{x \in X} \left( \frac{1}{m_{\mathcal{U}_k} (x)}\sum_{U \in \mathcal{U}_k, x \in U} g(d(x,X \setminus U))  - \frac{1}{m_{\mathcal{U}_k} (x)}\sum_{U \in \mathcal{U}_k, x \in U, d(x,X \setminus U) \leq T_k} g(d(x,X \setminus U))  \right) \geq $$
$$\geq 2g(T_k) - g(T_k) = g (T_k).$$
Therefore
$$\inf_{x \in X} \left( \frac{1}{m_{\mathcal{U}_k} (x)}\sum_{U \in \mathcal{U}_k, x \in U} (1+\frac{1}{k})^{d(x,X \setminus U)} \right) \geq \inf_{x \in X} \left( \frac{1}{m_{\mathcal{U}_k} (x)}\sum_{U \in \mathcal{U}_k, d(x,X \setminus U) \geq T_k} (1+\frac{1}{k})^{d(x,X \setminus U)} \right) \geq $$
$$\geq \inf_{x \in X} \left( \frac{1}{m_{\mathcal{U}_k} (x)}\sum_{U \in \mathcal{U}_k, d(x,X \setminus U) \geq T_k} g(d(x,X \setminus U)) \right) \geq g(T_k).$$
Since $\lim g (T_k) = \infty$ we get that
$$\lim_{k \rightarrow \infty} \inf_{x \in X} \left( \frac{1}{m_{\mathcal{U}_k} (x)}\sum_{U \in \mathcal{U}_k, x \in U} (1+\frac{1}{k})^{d(x,X \setminus U)} \right) = \infty,$$
so we are done by $(5)$ in theorem \ref{definitions}.

\end{proof}

\section{Connection with asymptotic dimension growth}

Recall that given a metric space $(X,d)$ and a cover $\mathcal{U}$ of $X$, $\lambda >0$ is called a Lebesgue number of $\mathcal{U}$ if for every $A \subset X$ with $diam (A) \leq \lambda$ there is $U \in \mathcal{U}$ such that $A \subset U$. Denote by $L (\mathcal{U})$ the largest Lebesgue number of $\mathcal{U}$. Denote further the multiplicity of $\mathcal{U}$ as
$$m (\mathcal{U} ) = \sup_{x \in X} m (x).$$
In \cite{Dr} the asymptotic dimension function was defined as:
$$ad_X (\lambda ) = \min \lbrace m (\mathcal{U}) : \mathcal{U} \text{ is a uniformally bounded cover of } X \text{ such that } L(\mathcal{U}) \geq \lambda \rbrace-1.$$
A space $X$ has subexponential dimension growth if $ad_X (\lambda )$ grows subexponentially.
\begin{proposition}
If $X$ has subexponential dimension growth, then $X$ has asymptotically large depth. 
\end{proposition}

\begin{proof}
We assume without loss of generality that $X$ is unbounded. We'll use the definition given in theorem \ref{subexp function} to prove that $X$ has asymptotically large depth. \\
Take $g (t) = t(ad_X ( 2 t)+1)$, since $X$ has subexponential dimension growth we get that $g (t)$ is monotone increasing and grows subexponentially. For every $k \in \mathbb{N}$ let $\mathcal{U}_k$ be a uniformly bounded cover of $X$ such that 
$L (\mathcal{U}_k ) \geq k$ and $m(\mathcal{U}_k) = ad_X (k) +1$. 
Notice that $L(\mathcal{U}_k) \geq k$ implies that for every $x \in X$ there is $U \in \mathcal{U}_k$ such that $d(x, X\setminus U) \geq \frac{k}{2}$. This implies that for every $x \in X$ we have:
$$ \frac{1}{m_{\mathcal{U}_k} (x)}\sum_{U \in \mathcal{U}_k, x \in U} g(d(x,X \setminus U))  \geq \frac{1}{m (\mathcal{U}_k)} g (\frac{k}{2} ) = \frac{ad_X (k) + 1}{m (\mathcal{U}_k)} \dfrac{k}{2}  = \dfrac{k}{2}.$$
Therefore
$$\lim_{k \rightarrow \infty} \inf_{x \in X} \left( \frac{1}{m_{\mathcal{U}_k} (x)}\sum_{U \in \mathcal{U}_k, x \in U} g(d(x,X \setminus U)) \right) =\infty,$$
and we are done.
\end{proof}

\begin{corollary}
Let $X$ be a discrete metric space of bounded geometry. If $X$ has sub exponential dimension growth, then $X$ has property A. 
\end{corollary}

\begin{proof}
Combine the above proposition with theorem \ref{property A}.
\end{proof}

\begin{remark}
The above corollary was already proven in \cite{Oz} using a different averaging method.
\end{remark}

\section{Invariance under Quasi-isometry}

First let us recall the following definitions and facts about coarse maps and quasi-isometries.

\begin{definition}
Let $\phi : X \rightarrow Y$ a map of metric space. \\
$\phi$ is called \textit{effectively proper}, if for every $R>0$ there is $S (R) >0$ such that for every $x \in X$, $\phi^{-1} (B (f(x) , R)) \subseteq B (x, S (R) )$. \\
$\phi$ is called \textit{bornologous} if there is a map $\rho : \mathbb{R}_{\geq 0} \rightarrow  \mathbb{R}_{\geq 0}$ such that for every $x, x' \in X$ the following holds:
$$ d_Y (\phi (x), \phi (x') ) \leq \rho (d_X (x, x')).$$
$\phi$ is called a \textit{coarse embedding} if it is both effectively proper and bornologous. \\
$\phi$ is called \textit{large scale Lipschitz} if there are constants $A \geq 1, C \geq 0$ such that for every $x, x' \in X$ the following holds:
$$ d_Y (\phi (x), \phi (x') ) \leq A d_X (x, x') +C.$$ 
$\phi$ is called a \textit{quasi-isometric embedding} if there are constants $A \geq 1, C \geq 0$ such that for every $x, x' \in X$ the following holds:
$$\frac{1}{A} d_X (x,x') - C \leq d_Y (\phi (x),\phi (x')) \leq A d_X (x,x') + C .$$
$X$ and $Y$ are said to be \textit{quasi-isometric} if there are quasi-isometric embeddings $\phi :X \rightarrow Y, \overline{\phi} : Y \rightarrow X$ such that
$$ \sup_{x \in X} d_X (x, \overline{\phi} (\phi (x))) < \infty,  \sup_{y \in Y} d_Y (y,  \phi (\overline{\phi} (y))) < \infty .$$
\end{definition}

\begin{lemma}
\label{length spaces}
Let $X$ be a length space and $Y$ be any metric space. Then $\phi : X \rightarrow Y$ is bornologous if and only if it is large scale Lipschitz.
\end{lemma}

\begin{proof} 
See lemma 1.10 in \cite{Roe}.
\end{proof}

\begin{lemma}[Milnor-Schwarz lemma]
\label{MS lemma}
Let $X$ be a proper geodesic metric space and let $\Gamma$ be group acting by isometries on $X$ such that the action is cocompact and proper, then $\Gamma$ is finitely generated and quasi-isometric (with respect to the word metric) to $X$.
\end{lemma}

\begin{proof}
See \cite{GH}[proposition 10.9].
\end{proof}

Next, we'll show the following:

\begin{proposition}
Let $X,Y$ metric spaces and $\phi: X \rightarrow Y$ that is effectively proper and large scale Lipschitz. If $Y$ has asymptotically large depth then $X$ has asymptotically large depth.
\end{proposition}

\begin{proof}
Let $A,C$ be the constants of the $\phi$ and is the definition of large scale Lipschitz. For $\mathcal{U}$ that is a cover of $Y$, define
$$\phi^{-1} (\mathcal{U}) = \lbrace \phi^{-1} (U) : U \in \mathcal{U} \rbrace .$$
Note that $\phi^{-1} (\mathcal{U})$ is a cover of $X$ and if $R$ is a bound on the diameter of sets in $\mathcal{U}$, then $S (R)$ is a bound on the diameter of the sets in $\phi^{-1} (\mathcal{U})$ and in particular, if $\mathcal{U}$ is uniformly bounded then $\phi^{-1} (\mathcal{U})$ is uniformly bounded. Also, note that for every $x \in X$ we have $m_{\phi^{-1} (\mathcal{U})} (x) \leq m_{\mathcal{U}} (\phi (x))$. Last, note that if $x \in \phi^{-1} (U), x' \notin \phi^{-1} (U)$ ,then $\phi (x) \in U, \phi (x') \notin U$ and therefore
$$d_Y (\phi (x), Y \setminus U) \leq d_Y ( \phi (x), \phi (x')) \leq A d_X (x,x') +C.$$
Since this is true for every $x \in \phi^{-1} (U), x' \notin \phi^{-1} (U)$, this yields that for every $x \in \phi^{-1} (U)$, we have that
$$\dfrac{d_Y (\phi (x), Y \setminus U) - C}{A} \leq d_X (x , X \setminus \phi^{-1} (U)) .$$
After all those observations, we are ready to prove the proposition using the definition given in theorem \ref{subexp function}: $Y$ has asymptotically large depth if and only if there is a monotone increasing function $g_Y : \mathbb{R}_{\geq 0} \rightarrow  \mathbb{R}_{+}$ with subexponential growth and a sequence of covers $\lbrace \mathcal{U}_k\rbrace_{k \in \mathbb{N}}$ of $Y$ such that:
\begin{enumerate}
\item  For each $k$, $\mathcal{U}_k$ is uniformly bounded and locally finite.
\item 
$$\lim_{k \rightarrow \infty} \inf_{y \in Y} \left( \frac{1}{m_{\mathcal{U}_k} (y)}\sum_{U \in \mathcal{U}_k, y \in U} g_Y (d_Y (y,Y \setminus U)) \right) =\infty.$$
\end{enumerate}
Denote now $g_X (t) = g_Y (At+C)$. $g_X$ is monotone increasing as a composition of monotone increasing functions and for every $b>1$, we have
$$ \lim_{ t \rightarrow \infty} \dfrac{g_X (t)}{b^t} = \lim_{ t \rightarrow \infty} \dfrac{g_Y (At+C)}{b^t} = \lim_{ t \rightarrow \infty} b^{\frac{C}{A}} \left( \dfrac{g_Y ( At+C)}{(b^\frac{1}{A})^{At+C}} \right) = 0 ,$$
and therefore $g_X (t)$ has subexponential growth. To finish, notice that by the above observations, $\phi^{-1} (\mathcal{U}_k)$ is uniformly bounded and locally finite for each $k$ and 
$$  \inf_{x \in X} \left( \frac{1}{m_{\phi^{-1} (\mathcal{U}_k)} (x)} \sum_{\phi^{-1} (U) \in \phi^{-1} (\mathcal{U}_k), x \in \phi^{-1} (U)} g_X (d_X (x,X \setminus \phi^{-1} (U))) \right) \geq$$
$$\geq  \inf_{x \in X} \left( \frac{1}{m_{\mathcal{U}_k} (\phi (x))}\sum_{U \in \mathcal{U}_k, \phi (x) \in U} g_X (\dfrac{d_Y (\phi (x), Y \setminus U) - C}{A}) \right) \geq$$
$$\geq \inf_{x \in X} \left( \frac{1}{m_{\mathcal{U}_k} (\phi (x))}\sum_{U \in \mathcal{U}_k, \phi (x) \in U} g_Y (d_Y (\phi (x), Y \setminus U)) \right) \geq$$
$$\geq \inf_{y \in Y} \left( \frac{1}{m_{\mathcal{U}_k} (y)}\sum_{U \in \mathcal{U}_k, y \in U} g_Y (d_Y (y,Y \setminus U)) \right) .$$
Therefore
$$\lim_{k \rightarrow \infty}  \inf_{x \in X} \left( \frac{1}{m_{\phi^{-1} (\mathcal{U}_k)} (x)} \sum_{\phi^{-1} (U) \in \phi^{-1} (\mathcal{U}_k), x \in \phi^{-1} (U)} g_X (d_X (x,X \setminus \phi^{-1} (U))) \right) = \infty, $$
and we are done.
\end{proof}

The above proposition has the following corollaries:

\begin{corollary}
Let $X,Y$ metric spaces and $\phi: X \rightarrow Y$ a quasi-isometric embedding. If $Y$ has asymptotically large depth then $X$ has asymptotically large depth.
\end{corollary}

\begin{proof}
Obvious.
\end{proof}

From this corollary we get:

\begin{corollary}
Asymptotically large depth is a quasi-isometry invariant, i.e., if $X$ and $Y$ are quasi-isometric then $X$ has asymptotically large depth if and only if $Y$ has asymptotically large depth.
\end{corollary}

Combine the above corollary with lemma \ref{MS lemma} to get:
\begin{corollary}
Let $X$ be a proper geodesic metric space and let $\Gamma$ be group acting by isometries on $X$ such that the action is cocompact and proper, then $\Gamma$ has asymptotically large depth if and only if $X$ has asymptotically large depth.
\end{corollary}

Also, by using lemma \ref{length spaces} and the above proposition, we get:

\begin{corollary}
Let $X,Y$ metric spaces such that $X$ is a length space and let $\phi: X \rightarrow Y$ be a coarse map. If $Y$ has asymptotically large depth, then $X$ has asymptotically large depth.
\end{corollary}

\bibliographystyle{alpha}
\bibliography{bibl}

\begin{thebibliography}{GdlH91}

\bibitem[Dra06]{Dr}
A.~N. Dranishnikov.
\newblock Groups with a polynomial dimension growth.
\newblock {\em Geom. Dedicata}, 119:1--15, 2006.

\bibitem[GdlH91]{GH}
{\'E}tienne Ghys and Pierre de~la Harpe.
\newblock Infinite groups as geometric objects (after {G}romov).
\newblock In {\em Ergodic theory, symbolic dynamics, and hyperbolic spaces
  ({T}rieste, 1989)}, Oxford Sci. Publ., pages 299--314. Oxford Univ. Press,
  New York, 1991.

\bibitem[Gro93]{Grom}
M.~Gromov.
\newblock Asymptotic invariants of infinite groups.
\newblock In {\em Geometric group theory, {V}ol.\ 2 ({S}ussex, 1991)}, volume
  182 of {\em London Math. Soc. Lecture Note Ser.}, pages 1--295. Cambridge
  Univ. Press, Cambridge, 1993.

\bibitem[HR00]{HR}
Nigel Higson and John Roe.
\newblock Amenable group actions and the {N}ovikov conjecture.
\newblock {\em J. Reine Angew. Math.}, 519:143--153, 2000.

\bibitem[Oza12]{Oz}
Narutaka Ozawa.
\newblock Metric spaces with subexponential asymptotic dimension growth.
\newblock {\em Internat. J. Algebra Comput.}, 22(2):1250011, 3, 2012.

\bibitem[Roe03]{Roe}
John Roe.
\newblock {\em Lectures on coarse geometry}, volume~31 of {\em University
  Lecture Series}.
\newblock American Mathematical Society, Providence, RI, 2003.

\bibitem[Wil09]{Will}
Rufus Willett.
\newblock Some notes on property {A}.
\newblock In {\em Limits of graphs in group theory and computer science}, pages
  191--281. EPFL Press, Lausanne, 2009.

\bibitem[Yu98]{Yu2}
Guoliang Yu.
\newblock The {N}ovikov conjecture for groups with finite asymptotic dimension.
\newblock {\em Ann. of Math. (2)}, 147(2):325--355, 1998.

\bibitem[Yu00]{Yu}
Guoliang Yu.
\newblock The coarse {B}aum-{C}onnes conjecture for spaces which admit a
  uniform embedding into {H}ilbert space.
\newblock {\em Invent. Math.}, 139(1):201--240, 2000.

\end{thebibliography}

\end{document}